\documentclass[a4paper]{amsart}
\usepackage[usenames,dvipsnames]{xcolor}
\usepackage[utf8]{inputenc}
\usepackage[T1]{fontenc}
\usepackage{lmodern}      	% without that, the fontenc T1 looks ugly
\usepackage{empheq}		
\usepackage{amssymb}  		
\usepackage{euscript}
\usepackage[pagebackref,colorlinks=true,linkcolor=blue,citecolor=red]{hyperref}
\renewcommand*{\backref}[1]{}
\renewcommand*{\backrefalt}[4]{({%
		\ifcase #1 Not cited.%
		\or On p.~#2%
		\else On pp.~#2%
		\fi%
	})} 
\usepackage[nameinlink,capitalise,noabbrev]{cleveref}
\crefname{subsection}{Subsection}{Subsection}
\crefname{equation}{Diagram}{Diagram}
\usepackage{todonotes}
\usepackage{xspace}
\usepackage{tikz-cd}
\newcommand{\hyph}{\EuScript{H}}
\newcommand{\bbc}{\mathbb{C}}
\newcommand{\imp}{\Rightarrow}

\newcommand{\fa}{\forall}

\newcommand{\sfop}[1]{\operatorname{\mathsf{#1}}}

\newcommand{\angs}[1]{\langle #1\rangle}

\newcommand{\cMono}{\sfop{cMono}}

\renewcommand{\tt}{\sfop{t\!t}}
\newcommand{\ff}{\sfop{f\!f}}
\newcommand{\icat}{$\infty$-category\xspace}
\newcommand{\icats}{$\infty$-categories\xspace}

\newcommand{\id}{\sfop{id}}				% identity map
\newcommand{\Fun}{\sfop{Fun}}
\newcommand{\adj}{\dashv}
\newcommand{\Sub}{\mathsf{Sub}}
\newcommand{\Core}{\sfop{Core}}

\newcommand{\Ho}{\sfop{Ho}}
\newcommand{\op}{^\mathsf{op}}				% for opposite category

\newcommand{\lucite}[1]{\cite[#1]{lurie2009higher}}

\newcommand{\lcc}{locally Cartesian closed\xspace}
\newcommand{\ie}{i.e.\@\xspace}
\newcommand{\eg}{e.g.\@\xspace}

\newcommand{\SLat}{\sfop{SLat}}
\newcommand{\HSLat}{\sfop{HSLat}}

\newcommand{\emb}{\rightarrowtail}

\newcommand{\isContr}{\sfop{isContr}}
\newcommand{\C}{\EuScript{C}}

\renewcommand{\S}{\EuScript{S}}
\newcommand{\E}{\EuScript{E}}

\newcommand{\hide}[1]{}
\renewcommand{\ker}{\sfop{ker}}
\newcommand{\Map}{\sfop{Map}}
\newcommand{\slc}[2]{{#1}_{/#2}}
\newcommand{\Str}{\S^{\sf tr}}
\numberwithin{equation}{section}
\newtheorem{theorem}[equation]{Theorem}
\newtheorem{lemma}[equation]{Lemma}
\newtheorem{proposition}[equation]{Proposition}
\newtheorem{corollary}[equation]{Corollary}

\theoremstyle{definition}
\newtheorem{definition}[equation]{Definition}
\newtheorem{example}[equation]{Example}

\newtheorem{remark}[equation]{Remark}

\newtheoremstyle{TheoremNum}
{}{}              %%% space between body and thm
{\itshape}                      %%% Thm body font
{}                              %%% Indent amount (empty = no indent)
{\bfseries}                     %%% Thm head font
{.}                             %%% Punctuation after thm head
{ }                             %%% Space after thm head
{\thmname{#1}\thmnote{ \bfseries #3}}%%% Thm head spec
\theoremstyle{TheoremNum}
\newtheorem{thmn}{Theorem}
\newtheorem{exm}{Example}

\title{Constructing Coproducts in Locally Cartesian Closed  $\infty$-Categories}
\author{Jonas Frey and Nima Rasekh}

\address{Department of Philosophy, Carnegie Mellon University, 5000 Forbes Avenue, Pittsburgh, PA 15213, USA
}
\email{jonasf@andrew.cmu.edu}

\address{{\'E}cole Polytechnique F{\'e}d{\'e}rale de Lausanne, SV BMI UPHESS, Station 8, CH-1015 Lausanne, Switzerland}

\email{nima.rasekh@epfl.ch}



\begin{document}
	
	\begin{abstract}
		We prove that every locally Cartesian closed  \icat with a subobject classifier
		has a \emph{strict initial object} and \emph{ disjoint and universal binary
			coproducts}.
	\end{abstract}
	
	\maketitle
	
	\section{Introduction}
	
	\subsection*{Elementary Toposes and Finite Colimits}
	
	{\it Categorical logic} uses results and constructions from category theory to
	study {\it type theory}, {\it set theory} and other concepts in mathematical
	logic. One key concept in categorical logic is that of an {\it elementary
		topos}. Elementary toposes admit a natural interpretation of {\it higher-order
		logic} \cite[Chapter~D4]{elephant2}, and also give rise to models of set
	theories~\cite{maclanemoerdijk1994topos,joyal1995algebraic}.  
	
	Elementary
	toposes were defined by Lawvere and Tierney as a generalization of
	\emph{Grothendieck toposes}. The latter always admit small limits and colimits
	since they are defined as categories of sheaves and are therefore locally
	presentable~\cite{Artin1972Seminaire}. Hence, the first definitions of elementary topos
	assumed the existence of both finite limits and finite colimits
	\cite{lawvere1970quantifiers,tierney1972sheaf}. However, it was soon realized
	that the existence of finite colimits could in fact be deduced from the other
	axioms and concretely that we have the following result: every finitely
	complete Cartesian closed category with a subobject classifier has finite
	colimits~\cite{mikkelsen1972finite,pare1974colimits,mikkelsen1976lattice}. 
	
	The recent decades have witnessed significant advances in the study of {\it
		homotopy invariant} mathematics. In particular, there is now a well developed
	theory of homotopy invariant categories, known as {\it $(\infty,1)$-categories}
	or simply {\it $\infty$-categories} \cite{bergner2010survey}, which have been
	used extensively in many areas relevant to homotopy theory, such as {\it
		homotopy coherent algebraic structures} or {\it derived geometry}
	\cite{lurie2017ha}.

	The theory of  Grothendieck toposes has successfully been generalized to the
	higher categorical setting -- both in the context of model categories
	\cite{rezk2010toposes} and \icats \cite{lurie2009higher} -- giving rise to the
	notion of (Grothendieck-)\emph{$\infty$-topos}.
	
	At the same time, categorical logicians have devised a homotopy invariant
	interpretation of \emph{Martin-Löf type theory}~\cite{martin1984intuitionistic},
	known as \emph{homotopy type theory}~\cite{hottbook}. This interpretation was
	quickly conjectured to generalize from homotopy types to arbitrary
	$\infty$-toposes, and a complete proof of this fact has recently been
	given~\cite{shulman2019all}.
	
	Just as the interpretation of higher order logic in $1$-toposes, the
	interpretation of type theory in $\infty$-toposes does not rely on the
	(co)completeness of the topos, which suggested to formulate a notion of
	`finitary' or `elementary' $\infty$-topos as natural target for the
	interpretation of type theory, analogous to Lawvere and Tierney's elementary
	$1$-toposes. Concrete proposals for a definition of elementary $\infty$-topos
	were given in \cite{shulman2017elementary, rasekh2018EHT}, and similarly to the
	first definitions of elementary $1$-topos, these definitions explicitly
	postulate the existence of finite colimits.
	
	This leaves us with the question whether we can recover finite colimits from the
	remaining axioms just as in the $1$-dimensional case. In the present paper we
	give a partial answer, by proving the following main result.
	
	\begin{thmn}[\ref{cor:descent coproducts}]
		Let $\{A_k\}_{k \in I}$ be a finite family of objects in a locally Cartesian
		closed \icat $\C$ with subobject classifier. Then the coproduct $\coprod_{k \in
			I} A_k$ exists, and pullback along the inclusion maps $i_k: A_k \to \coprod_{k
			\in I} A_k$ give rise to an equivalence of $\infty$-categories
		$$(i_k^*)_{k \in I}:\C_{/\coprod_{k \in I} A_k} \to \prod_{k \in I} \C_{/A_k}.$$
	\end{thmn}
	
	This result can be reformulated as saying that $\C$ admits a \emph{strict}
	initial object and \emph{disjoint and universal} binary coproducts. Of these
	properties, universality~\cite[Definition 6.1.1.2]{lurie2009higher} and
	strictness say that the respective colimits are preserved by pullback functors,
	which is a direct consequence of local Cartesian closure. \emph{Disjointness} of
	binary coproducts says that the commutative squares 
	\begin{equation*}
		\begin{tikzcd}[sep = small]
			A
			\ar[r,""']
			\ar[d,""]
			&	A
			\ar[d,""]
			\\	A
			\ar[r,""]
			&	A+B
		\end{tikzcd}
		\qquad\qquad
		\begin{tikzcd}[sep = small]
			0
			\ar[r,""']
			\ar[d,""]
			&	B
			\ar[d,""]
			\\	A
			\ar[r,""]
			&	A+B
		\end{tikzcd}
	\end{equation*}
	are pullbacks for all objects $A$, $B$, and the combination of universality and
	disjointness is the special case of Rezk's {\it descent condition} \cite[6.5]{rezk2010toposes} for binary coproducts. In the context of $1$-categories,
	descent for coproducts is also known as {\it extensivity}
	\cite{clw1993extensive}.
	
	\subsection*{What about pushouts?}
	
	Having settled the issue of coproducts, the remaining question is that of 
	pushouts and coequalizers. 
	
	However, it turns out that unlike the $1$-categorical situation, assuming the
	existence of a subobject classifier in fact does not suffice to prove the
	existence of pushouts in locally Cartesian closed $\infty$-categories as we
	illustrate via the following example.
	
	\begin{exm}[\ref{ex:counter}]
		Let $\Str$ be the full subcategory of the $\infty$-category $\S$ of spaces
		spanned by truncated spaces. Then $\Str$ is locally Cartesian closed and the
		discrete space $1\!+\!1$ is a subobject classifier. However, the diagram 
		\begin{center}
			\begin{tikzcd}[row sep=0.5in, column sep=0.5in]
				1 & S^1 \arrow[r] \arrow[l] & 1
			\end{tikzcd}
		\end{center}
		does not have a pushout.
	\end{exm}
	
	We can in fact give a more conceptual argument why it is possible to recover
	coproducts from the subobject classifier but not pushouts: the universal
	property of coproducts in \icats only depends on the homotopy types of the
	mapping spaces,  since the diagram used for coproducts is {\it discrete} and so
	cannot involve any higher homotopies. On the other side the diagram used to
	construct pushouts $(\bullet \leftarrow \bullet \rightarrow \bullet )$ is not
	discrete which means that the universal property of pushouts necessarily
	involves the notion of \emph{homotopy coherent diagram}~\lucite{Section~1.2.6}.
	
	Hence, it remains to determine what precise conditions we need to add to a
	locally Cartesian closed  $\infty$-category with a subobject classifier to be
	able to construct all finite colimits. The current hope is that we can obtain
	this result by additionally assuming the existence of universes.
	
	\subsection*{Structure of the paper}
	
	\cref{sec:lcc} recalls basic facts about locally Cartesian closed
	$\infty$-categories, including the Beck-Chevalley condition (\cref{lem:BC}),
	truncation levels (\cref{suse:truncation}), and the \emph{object of
		contractibility} (\cref{suse:contractibility}) -- a technique which allows to
	reduce contractibility questions to contractibility of subterminals.
	In \cref{sec:soc} we discuss subobject lattices and subobject classifiers, and
	show that if a locally Cartesian closed $\infty$-category has a subobject
	classifier, then its subobject lattices have finite joins (\cref{thm:sub-join}).
	Using this, we show in \cref{sec:initial-objects} that any locally Cartesian
	closed $\infty$-category with a subobject classifier has an initial object
	(\cref{cor:omega-lccc-has-init}), and in \cref{sec:bin-coprods} that it has
	disjoint binary coproducts (\cref{the:coproducts}).
	We conclude in \cref{sec:eitop} by discussing the relevance
	of our result to the notion of `elementary $\infty$-topos'. 
	
	\subsection*{\texorpdfstring{$\infty$}{oo}-Categorical Conventions} 
	
	In this paper we use $\infty$-categorical language and results via the model of
	\emph{quasi-categories} as developed in~\cite{joyal2008notes} and
	\cite{lurie2009higher}. However, the results proven here only rely on `model
	independent' properties of higher categories such as finite limits and locally
	Cartesian closure and so also hold analogously in any other
	$\infty$-cosmos~\cite{riehl2017fibrations}. 
	
	\subsection*{Acknowledgments}
	
	We thank the American Mathematical Society for running the Mathematics Research
	Communities Program in June, 2017, at which this work began, and the National
	Science Foundation for supporting the MRC program. 
	
	The second author would also like to thank the Max-Planck-Institut f{\"u}r
	Mathematik for its hospitality and financial support.
	
	\smallskip
	
	The first author acknowledges support by the Air Force Office of Scientific
	Research under award number
	% FA9550-15-1-0053 % MURI I 
	% and 
	FA9550-20-1-0305, % TOPOS
	and by the U.\ S.\ Army Research Office under grant number
	W911NF-21-1-0121. % COHESION

	\section{Some Facts about Locally Cartesian Closed
		\texorpdfstring{$\infty$}{oo}-Categories}\label{sec:lcc}
	Let $\C$ be an $\infty$-category with finite limits. Then for every morphism $f:
	A \to B$, the pullback functor $f^*: \C_{/B} \to \C_{/A}$ has a left adjoint
	$f_!: \C_{/A} \to \C_{/B}$ given by post-composition. If $f^*$ furthermore has a
	\emph{right} adjoint $f_*:\C_{/A} \to \C_{/B}$ for all $f$, then $\C$ is called
	{\it locally Cartesian closed}. If $B$ is the terminal object, we informally
	identify $\C$ with $\C_{/1}$ (see~\lucite{1.2.12.4}) and simply write $A_!\adj
	A^* \adj A_*$ for the adjoint string of functors along the terminal projection
	$A\to 1$.
	
	\begin{lemma}[Beck--Chevalley condition]\label{lem:BC}
		Given a pullback square
		\[
		\begin{tikzcd}[row sep = small]
			P   \ar[r,"h"']
			\ar[d,"k"'] 
			& A \ar[d,"f"] 
			\\B \ar[r,"g"]
			& C
		\end{tikzcd}
		\]
		in an $\infty$-category $\C$ with pullbacks, the canonical transformation
		\begin{equation*}
			h_!\circ k^* \to f^*\circ g_!
		\end{equation*}
		is an equivalence. If $\C$ is \lcc , then the canonical
		natural transformation 
		\begin{equation*}
			f^*\circ g_*\to h_*\circ k^*
		\end{equation*}
		is an equivalence.
	\end{lemma}
	\begin{proof}
		This is proven for the $\infty$-category $\S$ of spaces in \cite[Lemma
		2.1.6]{gepner2017operads}, but the proof only relies on $\C$ being locally
		Cartesian closed.
	\end{proof}
	
	Recall that an $\infty$-category $\C$ is called \emph{Cartesian closed} if it
	has finite products and for every $A\in\C$ the product functor $(-\times
	A):\C\to\C$ has a right adjoint commonly written ${(-)}^A:\C\to\C$ and called
	`exponentiation by $A$'. Every \lcc  $\infty$-category is
	Cartesian closed since $(-\times A)$ can be decomposed as $A_!\circ A^*$, and
	both $A_!$ and $A^*$ have right adjoints -- thus, exponentiation by $A$ is given
	by $A_*\circ A^*$ in this case. Since slices of \lcc 
	$\infty$-categories are obviously \lcc , we can conclude that
	all slices of \lcc  $\infty$-categories are Cartesian closed.\footnote{Conversely, every \icat with finite limits
		and Cartesian closed slices is locally cartesian closed -- the $1$-categorical
		proof of this statement given in~\cite[Corollary~A1.5.3]{elephant1} generalizes
		to \icats in a straightforward manner.}
	
	Moreover, we can deduce from the Beck--Chevalley condition that exponentiation
	commutes with pullback functors:
	
	\begin{lemma}\label{lem:reind-exp} Given morphisms $f:B\to A$, $g:C\to A$, and
		$h:D\to A$ in a \lcc  $\infty$-category $\C$ and $g,h\in
		\C_{/A}$, we have $f^*(h^g)\simeq (f^*h)^{f^*g}$.

	\end{lemma}
	\begin{proof}
		Form the pullback square 
		\begin{equation*}\begin{tikzcd}
				P
				\ar[r,"\overline{f}"']
				\ar[d,"f^*g"']
				&	C
				\ar[d,"g"]
				\\	B
				\ar[r,"f"]
				&	A
		\end{tikzcd}\end{equation*}
		of $g$ along $f$. We have 
		\begin{empheq}{align*}
			f^*(h^g) &\simeq (f^*\circ g_*\circ g^*)(h) \\
			&\simeq ((f^*g)_*\circ \overline{f}^*\circ g^*)(h) &\text{by the Beck--Chevalley condition}\\
			&\simeq ((f^*g)_*\circ (f^*g)^*\circ f^*)(h) &\text{since the square commutes}\\
			&\simeq (f^*h)^{f^*g}\,\,.
		\end{empheq}
	\end{proof}
	
	\subsection{Truncation and monomorphisms}\label{suse:truncation}
	For $n\geq -2$, recall that an object $A$ in an \icat $\C$ is called 
	\emph{$n$-truncated} if the mapping space $\Map_\C(X,A)$ is $n$-truncated
	for all objects $X\in\C$. The object is called \emph{contractible} or 
	\emph{terminal} if it is
	$(-2)$-truncated, and \emph{subterminal} if it is $(-1)$-truncated.
	
	An \emph{arrow} $f:A\to B$ in $\C$ is called {$n$-truncated} if for all $X\in\C$
	the postcomposition operation $\Map_\C(X,f):\Map_\C(X,A)\to\Map_\C(X,B)$ is an $n$-truncated map in
	$\S$, \ie if its fibers are $n$-truncated spaces. 
	If $\C$ has a terminal object $1$ then an object $A$ is $n$-truncated iff the
	morphism $A\to 1$ is $n$-truncated. Conversely, $f:A\to B$ is $n$-truncated as a
	morphism in $\C$ iff it is $n$-truncated as an object in 
	$\C_{/B}$.
	
	A morphism $f:A\to B$ is $(-2)$-truncated iff it is an equivalence. If $\C$ has
	pullbacks, then $f:A\to B$ is $(n+1)$-truncated iff its diagonal $\delta_f:A\to
	A\times_B A$ is $n$-truncated.

	Maps that are $(-1)$-truncated are also called \emph{monomorphisms}. Thus, $f:A\to B$ is an
	monomorphism  iff its diagonal $\delta_f: A \to A \times_B A$ is an equivalence, \ie the commutative
	square
	\[
	\begin{tikzcd}
		A		\ar[r,"\id"']
		\ar[d,"\id"'] 
		&	A		\ar[d,"f"]
		\\  A       \ar[r,"f"] 
		&   B
	\end{tikzcd}
	\]
	is a pullback.
	\begin{lemma} \label{lem:emb-shriek-corefl}
		Let $m:U\emb A$ be a monomorphism in an \icat $\C$ with 
		finite limits.
		\begin{enumerate} 
			\item For every $f:B\to U$, the commutative square 
			\(
			\begin{tikzcd}[sep = small]
				B	\rar["\id"']
				\dar["f"']
				& 	B	\dar["m\circ f"]
				\\  U 	\rar[tail,"m"]
				&	A 
			\end{tikzcd}
			\)
			is a pullback.
			\item\label{lem:emb-shriek-corefl-corefl}
			The adjunction $m_!\adj m^*$ is a coreflection, \ie its unit is an
			equivalence.
			\item\label{lem:emb-shriek-corefl-refl}
			If $\C$ is \lcc  then the adjunction $m^*\adj m_*$
			is a reflection, \ie its counit is an equivalence.
		\end{enumerate}
	\end{lemma}
	\begin{proof}
		The first claim follows from the pullback lemma since both small squares in the 
		following diagram are pullbacks.
		\begin{equation*}
			\begin{tikzcd}[row sep = small]
				B   \ar[r,"\id"]
				\ar[d,"f"']
				&   B   \ar[d,"f"]
				\\  U   \ar[r,"\id"]
				\ar[d,"\id"']
				&   U   \ar[d,"m"]
				\\  U   \ar[r,"m"]
				&   A
			\end{tikzcd}
		\end{equation*}
		The second claim follows from the first since the unit of $m_!\adj m^*$ at $f:U\to A$
		is the canonical map from $f$ to $m^*(m\circ f) $. The third claim follows from
		the second since the rightmost functor in an adjoint triple is fully faithful 
		iff
		the leftmost is.
	\end{proof}
	
	\begin{lemma}\label{lem:subterm-equiv} 
		Two subterminal objects $A$, $B$ in an $\infty$-category $\C$ are equivalent 
		whenever there exist 
		maps $f:A\to B$ and $g:B\to A$.
	\end{lemma}
	\begin{proof}
		This follows since all parallel maps into a subterminal are homotopic, in 
		particular every endomorphism is homotopic to the identity.
	\end{proof}
	
	\begin{lemma}\label{lem:section-mono}
		Let $A$ and $B$ be $0$-truncated
		objects in an $\infty$-category $\C$, and let $m: A\to B$, $e:B\to A$ such that
		$e\circ m=\id_A$ in $\Ho(\C)$. Then $m$ is a monomorphism.
	\end{lemma}
	\begin{proof}
		We give a proof in $\S$ (or in any $\infty$-category with finite limits), the 
		proof in general $\infty$-categories reduces to $\S$ by applying 
		corepresentable functors $\Map_\C(X,-)$.
		
		We have to show that $\delta_m:A\to 
		A\times_B A$ is an equivalence. This map may be viewed as a map in the slice 
		category over $A\times A$:
		\begin{equation*}\begin{tikzcd}
				A           \ar[rr,bend left,"m" description]
				\ar[dr,"\delta_A"',tail]
				\ar[r,dashed,"\delta_m" description]
				&           A\times_BA  \ar[r,""']
				\ar[d,"\ker(m)",tail]
				&	        B           \ar[d,"\delta_B",tail]
				\\&	        A\times A   \ar[r,"m\times m"]
				&	        B\times B
		\end{tikzcd}\end{equation*}
		Since $\delta_A$ and $\ker(m)$ are monomorphisms it is sufficient by 
		\cref{lem:subterm-equiv} to exhibit a map over $A\times A$ in the opposite direction of $\delta_m$. Such a map 
		is given by the mediating map in the following diagram
		\begin{equation*}\begin{tikzcd}
				A\times_BA  \ar[rr,]
				\ar[rd,dashed]
				\ar[rdd,tail,"\ker(m)"']
				&[-20pt]&   B           \ar[rd,"e",shift left = 1]
				\ar[dd,tail,"\delta_B" near end]
				\\[-15pt]&  A           \ar[rr,"\id"' near start, crossing over]
				\ar[d,"\delta_A",tail]
				&&          A           \ar[d,"\delta_A",tail]
				\\[+3pt]&   A\times A   \ar[r,"m\times m"]
				&           B\times B   \ar[r,"e\times e"]
				&           A\times A
		\end{tikzcd}\end{equation*}
		where the front rectangle is a pullback since $(e\times e)\circ(m\times m)\simeq\id$.
	\end{proof}
	
	\subsection{The object of contractibility}\label{suse:contractibility}
	
	Finally we will make use of the object of contractibility, motivated from 
	homotopy type theory.
	
	\begin{definition}\label{def:iscontr} 
		Given an object $A$ in a \lcc  $\infty$-category $\C$,
		we define the object $\isContr(A)$ by $\isContr(A)=A_!(\pi_*\delta_A)$, where 
		$\delta_A:A\to A\times A$ is the diagonal and $\pi:A\times A\to A$ is the 
		first projection.
	\end{definition}
	\begin{proposition}\label{prop:iscontr}
		Let $\C$ be a \lcc  $\infty$-category and let $A\in\C$.
		\begin{enumerate}
			\item \label{prop:iscontr-subterm}
			The object $\isContr(A)$ is always subterminal.
			\item \label{prop:iscontr-term}
			$A$ is terminal iff $\isContr(A)$ is terminal.
			\item \label{prop:iscontr-pullback}
			Given a second object $B$, we have $B^*(\isContr(A))\simeq
			\isContr(B^*A)$ in $\C_{/B}$.
		\end{enumerate}
	\end{proposition}
	\begin{proof}
		For 1,2 see \cite[Subsection 4.8]{rasekh2021nno}. 
		
		The third claim follows from the Beck--Chevalley condition for the pullback
		squares
		\begin{equation*}\begin{tikzcd}
				B\times A\times A
				\ar[r,""']
				\ar[d,""]
				&   B\times A
				\ar[r,""']
				\ar[d,""]
				&	B
				\ar[d,""]
				\\	A\times A
				\ar[r,""]
				&	A
				\ar[r,""]
				&	1
		\end{tikzcd}\end{equation*}
		together with the equivalences 
		\begin{equation*}
			\C_{/1}\simeq \C\qquad
			\C_{/(B\times A)}\simeq(\C_{/B})_{/{(B^*A)}}\qquad
			\C_{/(B\times A\times A)}\simeq (\C_{/B})_{/(B^*A\times B^*A)}
		\end{equation*}
		where we already commented on the first one, and the other two two are special cases of
		the dual of \lucite{2.1.2.5}.
	\end{proof}
	
	For more details on the object of
	contractibility in \lcc  $\infty$-categories see
	\cite[Subsection 4.8]{rasekh2021nno}.
	
	\section{Subobject Classifiers in \texorpdfstring{$\infty$}{oo}-Categories}
	\label{sec:soc}

	\subsection{Subobject lattices}  
	
	Let $\C$ be an \icat with pullbacks.
	The \emph{subobject lattice} 
	$\Sub(A)$ of an object $A$ in $\C$ is the full 
	subcategory of $\C_{/A}$ spanned by monomorphisms.
	Then $\Sub(A)$ is closed under finite limits in $\C_{/A}$, and since parallel maps 
	between subterminal objects are always homotopic it is (equivalent to 
	the nerve of) a poset, whence the finite limits are actually finite `meets' 
	(infima), \ie $\Sub(A)$ is a \emph{meet-semilattice}.
	
	If $\C$ is \lcc  then the 
	Cartesian closure of its slices $\C_{/A}$ is inherited by the subobject lattices 
	$\Sub(A)$ since exponentiation preserves truncatedness as a right adjoint. We
	shall refer to Cartesian closed posets as \emph{Heyting 
		semilattices}\footnote{This is a back-formation from the common term 
		\emph{Heyting algebra}, which in our terminology is a Heyting semilattice with finite 
		joins.}. The Cartesian exponentiation operation is called \emph{Heyting 
		implication} in the posetal case, and denoted $(-\imp-)$.
	
	For $f:B\to A$, the pullback functor $f^*:\C_{/A}\to\C_{/B}$ restricts to a 
	monotone and finite-meet-preserving map between subobject lattices.
	\[
	\begin{tikzcd}
		\Sub(B)\ar[d,hook]
		&
		\ar[l,"f^*"',dashed]
		\Sub(A)\ar[d,hook]\\
		\C_{/B}
		&
		\ar[l,"f^*"']
		\C_{/A}
	\end{tikzcd}
	\]
	If $\C$ is \lcc , then $f^*$ furthermore preserves Heyting
	implication by \cref{lem:reind-exp}, \ie it is a 
	\emph{morphism of Heyting semilattices}.
	
	Since homotopic maps in $\C$ induce equal maps between subobject lattices,
	the assignment $A\mapsto\Sub(A)$ is functorial \emph{on the homotopy category}, 
	i.e.\ it gives rise to a contravariant functor 
	\begin{equation}\label{eq:subfun}
		\Sub(-)\;:\;\Ho(\C)\op\;\to\;\HSLat
	\end{equation}
	into the category $\HSLat$ of Heyting semilattices and monotone maps 
	preserving finite meets and Heyting implication.
	
	The postcomposition maps $f_!:\C_{/B}\to\C_{/A}$ do not generally restrict to 
	subobject lattices (only if $f$ itself is a monomorphism), but if $\C$ is 
	\lcc then the right adjoints $f_*$ restrict to monomorphisms, so that for each 
	$f:B\to A$ the adjunction between slices restricts to an adjunction between 
	subobject lattices.
	\[
	\begin{tikzcd}[sep = large]
		\Sub(B)\ar[d,hook]
		\ar[r,"\fa_f"',bend right=18]
		\ar[r,phantom,"\bot"]
		&
		\ar[l,"f^*"',bend right=18]
		\Sub(A)\ar[d,hook]\\
		\C_{/B}
		\ar[r,"f_*"',bend right=18]
		\ar[r,phantom,"\bot"]
		&
		\ar[l,"f^*"',bend right=18]
		\C_{/A}
	\end{tikzcd}
	\]
	In other words, for each $f:B\to A$ in $\Ho(\C)$, the monotone map
	$f^*:\Sub(A)\to\Sub(B)$ has a right adjoint which we denote 
	$\fa_f:\Sub(B)\to\Sub(A)$. 
	
	By uniqueness of adjoints, this `universal quantification' operation gives 
	rise to a \emph{covariant} functor of type $\Ho(\C)\to\SLat$ with the same 
	object part as~\eqref{eq:subfun}.
	
	\subsection{Subobject classifiers} 
	
	Let $\C$ be again an \icat with pullbacks.
	We define $\cMono(\C)$ to be the non-full subcategory of the arrow category $\Fun(\Delta^1,\C)$
	with monomorphisms as objects and pullback squares as morphisms.
	Then the codomain projection $p : \cMono(\C)\to\C$ is a {right} fibration~\lucite{6.1.3.4}.
	Observe that for $A$ in $\C$, the fiber of $p$ over $A$ is a Kan complex which is equivalent
	to the underlying set of $\Sub(A)$. 
	We recall the following definition from~\lucite{6.1.6.1}.
	
	\begin{definition}
		A \emph{subobject classifier} in $\C$ is a terminal object in $\cMono(\C)$. 
	\end{definition}
	
	Thus, a subobject classifier is a monomorphism from which any other monomorphism can be obtained as a pullback in an essentially unique way.
	
	\begin{theorem}\label{thm:omega-trunc}
		Let $\tt:U\emb\Omega$ be a subobject classifier in an \icat $\C$ with
		pullbacks. Then $U$ is terminal and $\Omega$ is $0$-truncated.
	\end{theorem}
	\begin{proof}
		The object $\Omega$ is $0$-truncated because for every object $A$, the space 
		$\Map_\C(A,\Omega)$ is equivalent to the fiber of $p$ over $A$ and therefore
		to the underlying set of $\Sub(A)$, \ie $\Map(-,A)$ classifies the $0$-presheaf
		of subobjects. The object $U$ is terminal since it classifies \emph{maximal} subobjects. 
	\end{proof}
	
	\begin{lemma}\label{lem:not-retract}
		Let $m:A\to B$, $e:B\to A$ be maps in a \lcc 
		$\infty$-category $\C$ such that $e\circ m=\id_A$ in $\Ho(\C)$. Then given 
		$U\in\Sub(B)$, we have $\fa_e\,U\leq m^*\,U$ in $\Sub(A)$.
	\end{lemma}
	\begin{proof}
		By adjunction we have $U\leq \fa_m\,m^*\,U$, and therefore we can argue
		\begin{align*}
			\fa_e\,U\leq \fa_e\,\fa_m\,m^*\,U\leq \fa_{e\circ m}\,m^*\,U\leq m^*\,U
		\end{align*}
		by functoriality of $\fa$ on $\Ho(\C)$.
	\end{proof}
	\begin{theorem}\label{thm:sub-join} Let $\C$ be a \lcc 
		$\infty$-category with subobject classifier $\tt:U\emb\Omega$. Then for every
		object $A\in\C$ the poset $\Sub(A)$ has finite joins.
	\end{theorem}
	\begin{proof}
		Given $A\in\C$ we claim that a least element of $\Sub(A)$ is given by
		\begin{equation*}
			\bot = \fa_{\pi_1}\,\pi_2^*\,\tt,
		\end{equation*}
		where $A\xleftarrow{\pi_1}A\times\Omega\xrightarrow{\pi_2}\Omega$ is a product 
		span. Let $U\in\Sub(A)$, and let $f:A\to\Omega$ with $f^*\tt=U$. Then we have 
		\begin{align*}
			\bot=\fa_{\pi_1}\,\pi_2^*\,\tt 
			&\leq \angs{\id_A,f}^*\,\pi_2^*\,\tt &\text{by \cref{lem:not-retract}}\\
			&\leq f^*\tt&\text{by functoriality of $(-)^*$}\\
			&=U\,.
		\end{align*}
		Given $U,V\in\Sub(A)$ we claim that a binary join is given by
		\begin{equation*}
			U\vee V = 
			\fa_{\pi_1}\bigl((\pi_1^*U\imp\pi_2^*\tt)\wedge(\pi_1^*V\imp\pi_2^*\tt)\imp\pi_2^*\tt\bigr).
		\end{equation*}
		The derivation 
		\begin{align*}
			&\quad \pi_1^*U\wedge(\pi_1^*U\imp\pi_2^*\tt)
			\quad\leq\quad\pi_2^*\tt\\
			\imp&\quad\pi_1^*U\wedge(\pi_1^*U\imp\pi_2^*\tt)\wedge(\pi_1^*V\imp\pi_2^*\tt)
			\quad\leq\quad\pi_2^*\tt\\
			\Leftrightarrow&\quad \pi_1^*U
			\quad\leq\quad
			(\pi_1^*U\imp\pi_2^*\tt)\wedge(\pi_1^*V\imp\pi_2^*\tt)\imp\pi_2^*\tt\\
			\Leftrightarrow&\quad U
			\quad\leq\quad
			\fa_{\pi_1}\bigl((\pi_1^*U\imp\pi_2^*\tt)
			\wedge(\pi_1^*V\imp\pi_2^*\tt)\imp\pi_2^*\tt\bigr)
		\end{align*}
		shows that $U$ is indeed smaller than $U\vee V$, and similarly for $V$. To show 
		that $U\vee V$ is a least upper bound let $W\in\Sub(A)$ with $U\leq W$ and 
		$V\leq W$, and let $g:A\to \Omega$ with $g^*\tt=W$. Then we have 
		\begin{empheq}{align*}
			&U\vee V \\
			&= \fa_{\pi_1}\bigl((\pi_1^*U\imp\pi_2^*\tt)\wedge (\pi_1^*V\imp\pi_2^*\tt)
			\imp\pi_2^*\tt\bigr) \\ 
			&\leq \angs{\id_A,g}^*\bigl((\pi_1^*U\imp\pi_2^*\tt)
			\wedge(\pi_1^*V\imp\pi_2^*\tt)\imp\pi_2^*\tt\bigr)
			&\text{by \cref{lem:not-retract}} \\ 
			&=(U\imp W)\wedge(V\imp W)\imp W &\text{since $(-)^*$ preserves $\wedge,\imp$}\\
			&=W 
		\end{empheq}
	\end{proof}
	
	\begin{remark}
		The argument in the previous proof is well known from second order logic, and in
		its categorical incarnation from \emph{tripos theory} \cite{hjp80,pitts81} and 
		elementary topos theory~\cite{boileau1981logique}. It
		works in general whenever we have a presheaf $\hyph:\bbc\op\to\HSLat$ of Heyting
		semilattices on a $1$-category with finite products, such that
		\begin{enumerate} 
			\item reindexing maps along product projections have right adjoints, and 
			\item $\hyph$ has a \emph{generic predicate}, \ie the category of elements of
			the underlying presheaf of sets of $\hyph$ has a weakly terminal object.
		\end{enumerate}
		(Note that we do \emph{not} require a Beck-Chevalley condition.)
		
		From the point of view of locally Cartesian closed categories we point out that the construction applies exponentiation and pushforward functors $f_*$ only to \emph{subobjects} rather than general morphisms.
	\end{remark}
	
	\section{Initial Objects}\label{sec:initial-objects}

	In this section we prove that every \lcc  $\infty$-category
	with subobject classifier has a strict initial object. 
	\begin{definition}\label{def:init}
		An \emph{initial} object in an $\infty$-category $\C$ is an object $0$ such that
		$\Map_\C(0,A)$ is contractible for all $A\in\C$. The initial object is called
		\emph{strict}, if $\C_{/0}$ is equivalent to the terminal $\infty$-category.
	\end{definition}
	The following theorem gives a characterization of initial objects.
	\begin{theorem}\label{thm:initial-equiv} 
		Let $\C$ be a \lcc 
		$\infty$-category and $I$ an object of $\C$. Then the following are equivalent.
		\begin{enumerate}
			\item $I$ is initial in $\C$.
			\item $\C_{/I}$ is equivalent to the terminal $\infty$-category.
			\item $\Sub(I)$ is equivalent to the terminal preorder.
		\end{enumerate}
	\end{theorem}
	\begin{proof}
		Evidently (1) implies (3) since every subobject of an initial object has to be 
		trivial.
		
		Conversely, if $\Sub(I)\simeq 1$ then for any $X \to I$ the subobject $\isContr_I(X) \emb I$ is maximal, meaning that $X \to I$ is an
		equivalence. This shows that (3) implies (2).
		
		Finally, to show that $I$ is initial we have to show that the mapping space 
		$\Map_\C(I,X)$ is terminal for all $X\in\C$. Since $\Map_\C(I,X)\simeq\Map_\C(1,X^I)$ and
		$\Map_\C(1,-)$ preserves finite limits, it is enough to show that $X^I$ is
		terminal in $\C$. Since $X^I = \Pi_I I^* X$ and $\Pi_I:\C_{/I}\to\C$
		preserves limits, it is enough to show that $I^*X$ is terminal in $\C_{/I}$. 
		This follows from (2).
	\end{proof}
	\begin{remark}\label{rem:strict}
		Implication (1) to (2) of the theorem tells us in particular that 
		initial objects in locally Cartesian closed \icats are always \emph{strict} (\cref{def:init}).
	\end{remark}
	\begin{corollary}\label{cor:omega-lccc-has-init} Let $\C$ be a \lcc $\infty$-category with subobject classifier. Then $\C$ has a strict
		initial object. 
	\end{corollary}
	\begin{proof}
		By \cref{thm:sub-join}, the terminal object of $\C$ has a least subobject
		$0\emb 1$. Since any subobject of a least subobject is trivial we
		have $\Sub(0)\simeq 1$, and \cref{thm:initial-equiv} together with
		\cref{rem:strict} imply that $0$ is
		a strict initial object.
	\end{proof}
	
	\section{Binary Coproducts}\label{sec:bin-coprods}
	
	In this section we prove that every \lcc  $\infty$-category with subobject classifier $\Omega$ has finite coproducts by using the fact that the subobject lattices have finite joins (\cref{thm:sub-join}). To motivate our proof, we 
	start by discussing the $1$-categorical case.
	
	According to Johnstone~\cite[A2.2]{elephant1}, the first proofs of the existence 
	of finite colimits in elementary toposes were given by 
	Mikkelsen and 
	Paré~\cite{mikkelsen1972finite,pare1974colimits,mikkelsen1976lattice}.  
	Mikkelsen's proof does not seem to have been published. Par{\'e} proved -- using 
	Beck's theorem -- that in any elementary $1$-topos $\E$ the power object functor 
	$\Omega^{(-)}: \E\op \to \E$ is monadic, which implies that $\E\op$ has finite 
	limits as a category of Eilenberg-Moore algebras over a finite-limit category.
	
	Although there is an $\infty$-categorical analogue of Beck's theorem
	\cite[Theorem~4.7.3.5]{lurie2017ha}, this proof cannot be generalized as the
	corresponding functor of $\infty$-categories $\Omega^{(-)}: \C\op \to \C$ is not
	monadic and in fact not even conservative for the most simple examples:
	if $\C= \S$ then $\Omega = \{0,1\}$, the two element set, and the functor $\Omega^{(-)}: \S\op \to \S$
	takes every connected space to $\Omega$, and every map between connected spaces
	to an equivalence.
	
	Our proof of the existence of binary coproducts is based on an `internal-language proof' in $1$-toposes that avoids the monadicity theorem and was given as an 
	Exercise in~\cite[Exercise II.5.]{lambekscott86}. The idea is to 
	`carve out' the coproduct $A+B$ as subobject of $\Omega^A\times\Omega^B$. In 
	trying to adapt this proof to $\infty$-categories, we are met with two 
	obstacles:
	\begin{enumerate}
		\item While in a $1$-topos every object $A$ embeds into its power object 
		$\Omega^A$, this cannot work in higher toposes as, by \cref{thm:omega-trunc},
		$\Omega$ -- and therefore $\Omega^A$ and all its subobjects -- are 
		$0$-truncated. 
		\item To verify the universal property, the internal-language proof exhibits the unique arrow by first defining a (monic) binary relation,
		and then showing that it is single-valued and total. This kind of argument
		cannot work in the higher setting since it relies on the fact that the 
		graph $\angs{1,f}:A\to A\times B$ of a map $f:A\to B$ is always monic, which is 
		not the case \eg in $\S$.
	\end{enumerate}
	To overcome the first hurdle, we replace the $\Omega^A$ 
	in the construction with an object $\overline{A}$ known as 
	\emph{partial map classifier} or \emph{partial map representer}~\cite[pg.~101]{elephant1}
	in $1$-topos theory (\cref{lem:disj-emb}). To address the second point, we replace the classical internal-logic proof by an argument which is inspired by homotopy type theory (\cref{lem:coprod-from-join}),
	and which crucially relies on the technique of the object of contractibility, which we reviewed in \cref{suse:contractibility}.
	
	\begin{lemma}\label{lem:disj-emb}
		Let $A$ be an object in a \lcc  $\infty$-category $\C$ with
		subobject classifier $\tt:1\to\Omega$. Then there exists an object $\overline A$
		admitting disjoint monomorphisms of $A$ and $1$, \ie there exists a pullback
		square
		\begin{equation*}
			\begin{tikzcd}
				0
				\ar[r,tail]
				\ar[d,""',tail]
				&	1
				\ar[d,"",tail]
				\\	A
				\ar[r,"",tail]
				&	\overline{A}
		\end{tikzcd}\end{equation*}
		where all sides are monomorphisms and the upper left object is initial.
	\end{lemma}
	\begin{proof}
		Let $a:A\to 1$ be the terminal projection, and define 
		$(\overline{a}:\overline{A}\to\Omega) = \tt_*a$. Then by \cref{lem:emb-shriek-corefl}(\ref{lem:emb-shriek-corefl-refl}) we have 
		$\tt^*\overline{a}\simeq a$, \ie there is a pullback square 
		\begin{equation*}\begin{tikzcd}
				A
				\ar[r,tail]
				\ar[d,"a"]
				&	\overline{A}
				\ar[d,"\overline{a}"]
				\\	1
				\ar[r,"\tt",tail]
				&	\Omega
			\end{tikzcd}.
		\end{equation*}
		The lower map is a monomorphism by \cref{lem:section-mono}, and the
		upper map is a monomorphism by pullback stability.
		Now let 
		\begin{equation*}\begin{tikzcd}
				0
				\ar[r,tail,"e"']
				\ar[d,tail,"e"']
				&	1
				\ar[d,"\ff",tail]
				\\	1
				\ar[r,"\tt",tail]
				&	\Omega
		\end{tikzcd}\end{equation*}
		be the  classifying pullback square of the least subobject $0\emb 1$ of $1$, 
		such that $\ff:1\emb\Omega$ represents the truth value `false'. Again, $\ff$ 
		is a monomorphism by \cref{lem:section-mono}. The upper and left maps 
		can be chosen to be equal since $\Map_\C(0,1)$ is contractible.

		Forming the pullback in the arrow category $\Fun(\Delta^1,\C)$ we obtain a 
		commutative cube
		\begin{equation*}
			\begin{tikzcd}[sep = small]
				&   I       \ar[rr]
				\ar[dd,"i" pos=.2]
				\ar[dl]
				&&  J       \ar[dl,"k" pos=.7]
				\ar[dd,"j"]
				\\  A       \ar[dd,"a"']
				\ar[rr,crossing over]
				&&  \overline{A} 
				\\& 0       \ar[rr,"e"' pos=.2]
				\ar[dl,"e"' pos=.3]
				&&  1       \ar[ld,"\ff"]
				\\  1       \ar[rr,"\tt"]
				&&  \Omega  \ar[from=uu,crossing over,"\overline{a}" pos=.2]
			\end{tikzcd}
		\end{equation*}
		in which the left and right sides are pullbacks, since pullbacks are computed
		pointwise in functor categories. We already know that the front and bottom
		squares are pullbacks, and conclude that the remaining two are as well by the
		pullback lemma. The map $i$ is an equivalence since $\C_{/0}\simeq 1$ by
		\cref{thm:initial-equiv}. Furthermore we have 
		\begin{equation*}
			j \simeq \ff^*(\tt_*a)\simeq e_*(e^*a)\simeq e_*i
		\end{equation*}
		by the Beck--Chevalley condition (\cref{lem:BC}), which means that $j$ is
		an equivalence as well since terminal objects are preserved by right adjoints.
		
		Finally, $k$ is a monomorphism as a pullback of $\ff$ and the desired square 
		is recovered on the top of the cube.
	\end{proof}
	
	\begin{lemma}\label{lem:iscontr-subterm}
		Let $U,V\in\Sub(1)$ be subterminals in a \lcc 
		$\infty$-category $\C$, such that $U\vee V=\top$ in $\Sub(1)$. An object 
		$A\in\C$ is contractible whenever $U^*A$ is contractible in $\C_{/U}$ and 
		$V^*A$ is contractible in $\C_{/V}$.
	\end{lemma}
	\begin{proof}
		It is sufficient to show $\isContr(A)\geq U$ and $\isContr(A)\geq V$ in
		$\Sub(1)$, or equivalently that $U^*\isContr(A)\simeq 1$ and
		$V^*\isContr(A)\simeq 1$ in $\C_{/U}$ and $\C_{/V}$, respectively. This follows
		from the assumption together with \cref{prop:iscontr}(\ref{prop:iscontr-term})
		since we have $U^*(\isContr(A))\simeq\isContr(U^*A)$ and
		$V^*(\isContr(A))\simeq\isContr(V^*A)$ by
		\cref{prop:iscontr}(\ref{prop:iscontr-pullback}).
	\end{proof}
	
	\begin{lemma}\label{lem:coprod-from-join}
		Let $U\stackrel{i}{\rightarrowtail}A\stackrel{j}{\leftarrowtail}V$
		be a cospan of monomorphisms in a \lcc \icat $\C$, such that $U\wedge V$ is a
		least subobject of $A$, and $\top$ is a least upper bound of $U$ and $V$ in
		$\Sub(A)$. Then $i$ and $j$ exhibit $A$ as a disjoint coproduct of $U$ and $V$.
	\end{lemma}
	\begin{proof}
		Since the forgetful functor $A_!:\C_{/A}\to\C$ preserves coproducts as a left
		adjoint we may w.l.o.g.\ work in the slice category and thus assume that $A=1$.
		
		To show that we have a coproduct, we have to check that for all objects $X\in\C$ and arrows $f:U\to X$, $g:V\to X$,
		the pullback of the cospan
		\begin{equation*}\begin{tikzcd}
				&	\Map_\C(1,X)
				\ar[d,"\angs{\Map_\C(i,X),\Map_\C(j,X)}"]
				\\	1
				\ar[r,"\angs{f,g}"]
				&	\Map_\C(U,X)\times\Map_\C(V,X)
		\end{tikzcd}\end{equation*}
		in $\S$ is contractible. This cospan is equivalent to the image of the cospan 
		\begin{equation}\label{eq:cospan-c}
			\begin{tikzcd}
				&	X
				\ar[d,"\angs{c,d}"]
				\\	1
				\ar[r,"\angs{f,g}"]
				&	X^U\times X^V
		\end{tikzcd}\end{equation}
		under $\Map_\C(1,-)$, where $c$ and $d$ are exponential transposes of projection
		maps. Since $\Map_\C(1,-)$ preserves limits, it suffices to show that the pullback of
		the latter cospan is terminal in $\C$. By \cref{lem:iscontr-subterm} and
		since pullback functors preserve limits, it suffices to show that the images
		of~\eqref{eq:cospan-c} under $U^*$ and $V^*$ are contractible in $\C_{/U}$ and
		$\C_{/V}$, respectively. By symmetry, it is enough to consider the first case. We
		have 
		\begin{equation}\label{eq:usxu}
			U^*(X^U) = U^* (U_* (U^* X))\simeq U^*X
		\end{equation}
		since $U^*\adj U_*$ is a reflection (\cref{lem:emb-shriek-corefl}), and 
		by applying the Beck--Chevalley condition for the pullback square
		$
		\begin{tikzcd}[sep = 6]
			0
			\ar[r,"i"']
			\ar[d,"i"']
			&	V
			\ar[d,""]
			\\	U
			\ar[r,""]
			&	1
		\end{tikzcd}
		$
		we get 
		\begin{equation*}
			U^*(X^V) 
			= U^*(V_*(V^*X))
			\simeq i_* (i^* (V^* X)) 
			\simeq i_*1  
			\simeq 1,
		\end{equation*}
		since all objects over $0$ are terminal (\cref{thm:initial-equiv}).
		Furthermore one can show that modulo the equivalence~\eqref{eq:usxu} we have 
		$U^*(c)\simeq\id$, and since $U^*$ preserves limits we conclude
		\begin{equation*}
			U^*\left(
			\begin{tikzcd}[sep = small]
				&	X
				\ar[d,"\angs{c,d}"]
				\\	1
				\ar[r,]
				&	X^U\times X^V
			\end{tikzcd}
			\right) \simeq 
			\left(
			\begin{tikzcd}[sep = small]
				&	U^*X
				\ar[d,"\id"]
				\\	U^*1
				\ar[r]
				&	U^*X
			\end{tikzcd}
			\right).
		\end{equation*}
		The pullback of the right hand cospan is contractible in $\C_{/U}$ since 
		$U^*1$ is,  and equivalences are stable under pullback.
		
		Disjointness is clear since the injections are monic by assumption, and their
		pullback coincides with the meet $U\wedge V=\bot$ in $\Sub(A)$, which is initial
		by~\cref{thm:initial-equiv}.
	\end{proof}
	\begin{theorem} \label{the:coproducts}
		Let $\C$ be a \lcc  $\infty$-category with a subobject
		classifier. Then $\C$ has disjoint binary coproducts.
	\end{theorem}
	\begin{proof}
		Let $A$ and $B$ be objects of $\C$. By \cref{lem:coprod-from-join} it is
		sufficient to find an object $C$ admitting monomorphisms $A\emb C$ and $B\emb C$
		such that $A\wedge B=\bot$ and $A\vee B=\top$ in $\Sub(C)$.
		
		By \cref{lem:disj-emb} we have pullback squares
		\begin{equation*}
			\begin{tikzcd}
				0
				\ar[r,tail]
				\ar[d,""',tail]
				&	1
				\ar[d,"",tail]
				\\	A
				\ar[r,"",tail]
				&	\overline{A}
			\end{tikzcd}
			\qquad\qquad
			\begin{tikzcd}
				0
				\ar[r,tail]
				\ar[d,""',tail]
				&	1
				\ar[d,"",tail]
				\\	B
				\ar[r,"",tail]
				&	\overline{B}
			\end{tikzcd}
		\end{equation*}
		Forming the `transposed product'
		\[
		\Bigg(
		\begin{tikzcd}[sep = small]
			0 	\rar[tail]
			\dar[tail] 
			& 	1	\dar[tail]
			\\  A	\rar[tail]
			& 	\overline{A}
		\end{tikzcd}
		\Bigg)
		\times
		\Bigg(
		\begin{tikzcd}[sep = small]
			0 	\rar[tail]
			\dar[tail] 
			& 	B	\dar[tail]
			\\  1	\rar[tail]
			& 	\overline{B}
		\end{tikzcd}
		\Bigg)
		\;=\;
		\Bigg(
		\begin{tikzcd}[sep = small]
			0 	\rar[tail]
			\dar[tail] 
			& 	B	\dar[tail]
			\\  A	\rar[tail]
			& 	\overline{A}\times\overline{B}
		\end{tikzcd}
		\Bigg)
		\]
		of these two pullbacks yields a pullback square exhibiting $A$ and $B$ as 
		disjointly embedded in an object $\overline{A}\times\overline{B}$. The desired 
		cospan $A\emb C \leftarrowtail B$ is obtained by setting $C= A\vee B$ in
		$\Sub(\overline{A}\times\overline{B})$.
	\end{proof}
	
	The following summarizes all our results.
	
	\begin{theorem} \label{cor:descent coproducts}
		Let $\{A_k\}_{k \in I}$ be a finite family of objects in a \lcc \icat $\C$ with subobject classifier. Then the coproduct $\coprod_{k \in I} A_k$ exists, and the inclusion maps $i_k: A_k \to \coprod_{k \in I} A_k$ give rise to an equivalence of $\infty$-categories
		$$(i_k^*)_{k \in I}:\C_{/\coprod_{k \in I} A_k} \to \prod_{k \in I} \C_{/A_k}.$$
	\end{theorem}
	
	\begin{proof}
		If $I$ is empty, then this is precisely the statement that the initial object exists and is strict (\cref{cor:omega-lccc-has-init}). For $I$ non-empty, this is a direct consequence of the fact that coproducts exists and are disjoint (\cref{the:coproducts}) and universal, as $\C$ is \lcc  and left adjoints preserve colimits \cite[Proposition 5.2.3.5]{lurie2009higher}.
	\end{proof}
	
	\section{Coproducts and Pushouts in an Elementary
		\texorpdfstring{$\infty$}{oo}-Topos}\label{sec:eitop}
	
	In this final section we apply our result to the theory of elementary
	$\infty$-toposes. Following \cite{rasekh2018EHT,shulman2017elementary} we
	consider the following definition.
	
	\begin{definition}
		An elementary $\infty$-topos  is a finitely complete and cocomplete \lcc
		$\infty$-category $\E$ with a subobject classifier and enough
		universes\footnote{Here a \emph{universe} is an arrow $p :
			\mathcal{U}_*\to\mathcal{U}$ such that for all objects $A$ of $\E$ the induced
			map $\Map(A,\mathcal{U})\to\Core(\slc{\E}{A})$ is a monomorphism, and the class
			of pullbacks of $p$ satisfies certain closure conditions. For details
			see~\cite{rasekh2018EHT}.}.
	\end{definition}
	
	\cref{cor:descent coproducts} immediately gives us the following.
	
	\begin{corollary}
		An $\infty$-category $\E$ is an elementary $\infty$-topos if and only if it is
		\lcc  and has coequalizers, a subobject classifier, and
		enough universes.
	\end{corollary}
	
	This result moves us closer to the modern definition of elementary toposes, with
	the main difference being that we still assume the existence of coequalizers.
	The final question is whether we can construct coequalizers from the remaining
	axioms. 
	
	The following example shows that a subobject classifier certainly does not suffice to construct pushouts.
	
	\begin{example} \label{ex:counter} Let $\Str$ be the full subcategory of $\S$
		spanned by truncated spaces. Note that $\Str$ is \lcc and the discrete space
		$1\!+\!1$ is a subobject classifier. We claim that the diagram 
		\begin{equation}\label{eq:sphere-pushout}
			\begin{tikzcd}[row sep=0.5in, column sep=0.5in]
				1 & S^1 \arrow[r] \arrow[l] & 1
			\end{tikzcd}
		\end{equation}
		does not have a pushout in $\Str$. First, note that the pushout in $\S$ is just
		the $2$-sphere $S^2$. This implies that the $n$-truncation $\tau_{\leq n}S^2$ is
		the pushout of this diagram in the subcategory $\S^{\leq n}$ of $n$-truncated
		spaces. Now if~\eqref{eq:sphere-pushout} had a pushout $C$ in $\Str$ then the
		$n$-truncations of $C$ would \emph{also} be pushouts in $\S^{\leq n}$, which
		would imply that $\tau_{\le n} C \simeq \tau_{\le n}S^2 $ for all $n\geq 0$.
		This is impossible since $S^2$ is not truncated~\cite{gray1969sphere}.
	\end{example}
	
	\bibliographystyle{alpha}
	\bibliography{bib}
\end{document}